\documentclass[a4paper]{article}
\usepackage{amsthm,amsfonts,amsmath,amssymb,units,bbold}
\usepackage[cp1251]{inputenc}
\usepackage[english]{babel}
\usepackage[final]{graphicx}
\usepackage{setspace}
\usepackage[12pt]{extsizes}
\begin{document}
\newcommand{\per}{{\rm per}}
\newtheorem{teorema}{Theorem}
\newtheorem{lemma}{Lemma}
\newtheorem{utv}{Proposition}
\newtheorem{svoistvo}{Property}
\newtheorem{sled}{Corollary}
\newtheorem{con}{Conjecture}
\newtheorem{zam}{Remark}
\newtheorem{const}{Construction}

\author{A. A. Taranenko \thanks{Sobolev Institute of Mathematics, Novosibirsk, Russia.
e-mail: taa@math.nsc.ru. The work is supported in part by Young Russian Mathematics award.}}
\title{Regularity and counting lemmas for multidimensional matrices}
\date{}
\maketitle
\begin{abstract}
In the present paper we propose generalizations of the regularity and counting lemmas for multidimensional matrices under a finite alphabet.  Firstly, we prove a variant of a multidimensional regularity lemma with the help of a translation of $\varepsilon$-regularity from graphs to matrices. Next, we state that this $\varepsilon$-regularity is sufficient for obtaining a matrix analogue of the counting lemma for $2$-dimensional matrices but not for higher-dimensional cases. Finally, we introduce $\varepsilon$-regular patterns that allow us to deduce a multidimensional counting lemma.
\end{abstract}

\section*{Introduction}

A regularity method is a powerful tool yielding many results in extremal combinatorics, especially, in extremal graph and hypergraph theory. The method is based on three statements that are known as regularity, counting, and removal lemmas. This paper aims to state analogues of the regularity and counting lemmas for $2$-dimensional matrices and matrices of higher dimensions.

For the sake of completeness and for revealing the similarities between graph and matrix results, we briefly remind the main concepts and lemmas of the standard regularity method.  

Let $G= (V,E)$ be a simple graph with a vertex set $V$ and an edge set $E$, and let $A$ and $B$ be disjoint vertex subsets. The density $\rho(A,B)$ between sets $A$ and $B$ is 
$$\rho(A,B) = \frac{e(A,B)}{|A||B|},$$
where $e(A,B)$ denotes the number of edges between $A$ and $B$.

Given $\varepsilon > 0$, the pair of sets $(A,B)$ is said to be $\varepsilon$-regular if for all $X \subset A$ and $Y \subset B$ such that $|X| \geq \varepsilon |A|$ and $|Y| \geq \varepsilon |B|$ it holds
$$|\rho(X,Y) - \rho(A,B)| \leq \varepsilon.$$

A partition  $V_0 \sqcup V_1 \sqcup  \ldots \sqcup V_k $  of the vertex set $V$ into disjoint subsets with the exceptional set $V_0$ is called a balanced partition if all $V_i$, except $V_0$, have the same size. 

Given $\varepsilon > 0 $, a balanced partition $V_0 \sqcup V_1 \sqcup  \ldots \sqcup V_k $ is said to be $\varepsilon$-regular if $|V_0| \leq \varepsilon |V|$  and all but at most $\varepsilon k^2$ pairs $(V_i, V_j)$ for $i \neq j$ are $\varepsilon$-regular.

The key result of the graph regularity method is the Szemer\'edi's regularity lemma. It firstly appeared in~\cite{semteor} as an auxiliary lemma for proving that long arithmetic progressions exist in any dense enough subsets of natural numbers. Later the lemma was separately stated in~\cite{semlem}.

\begin{teorema}[Szemer\'edi's regularity lemma] \label{mainlem}
For every $\varepsilon > 0 $  there is  $T = T(\varepsilon)$ such that any graph  $G = (V,E)$  with the vertex set of size at least  $T$ has an $\varepsilon$-regular partition $V_0 \sqcup V_1 \sqcup  \ldots \sqcup V_t $ with $t \leq T$.
\end{teorema}

The counting lemma for graphs is the following quite simple observation.

\begin{lemma}[Counting lemma] \label{countgr}
For all  $t \in \mathbb{N}$, $\rho > 0$ and $\delta > 0$ there is $\varepsilon = \varepsilon(t,\rho,\delta)$ such that the following holds. If $G$ is $t$-partite graph with parts $V_1, \ldots, V_t$, $|V_i| = n$ and each pair $(V_i, V_j)$ is $\varepsilon$-regular of density $\rho$ then the number of complete subgraphs $K_t$ in  $G$ is $(1 \pm \delta) \rho^{{t \choose 2}} n^t$.
\end{lemma}

Combination of the regularity and the counting lemma produces the graph removal lemma. One of the first forms of this lemma was established by Ruzsa and Szemer\'edi in~\cite{ruzsa}.

\begin{teorema}[Removal lemma]
For any graph $H$ on $m$ vertices and any $\varepsilon > 0$ there exists $\delta > 0$ for which the following holds.  If $G$ is a graph on $n$ vertices such that one needs to delete at least $\varepsilon n^2$ edges of $G$ to destroy all copies of $H$ in $G$ then the graph $G$ contains at least $\delta n^m$ copies of $H$.
\end{teorema}

The next stage in the development of the regularity method was an extension on hypergraphs.
The main obstacle in this way was finding a notion of $\varepsilon$-regularity for hypergraphs such that both the regularity lemma and the counting lemma hold. 

The most natural generalizations of the Szemer\'edi regularity lemma for hypergraphs are proved in papers~\cite{czy,prom} but in~\cite{conter} Nagle and R\"odl constructed an example of a hypergraph showing that these generalizations do not imply an analogue of the counting lemma.

An appropriate regularity lemma for hypergraphs was firstly proved by R\"odl and Skokan in~\cite{hypsem}, and later using a different technique another regularity lemma was obtained by Gowers in~\cite{hypsem2}. An accompanying counting and removal lemmas are proved in~\cite{hypcount}. Notably, all of these results required quite sophisticated notions of regularity for hypergraphs and tedious techniques.

The correspondence between graphs and their adjacency matrices (and between hypergraphs and their multidimensional adjacency matrices) allows us to aim for the regularity method for matrices.  Analogues of the Szemer\'edi's regularity lemma for $2$-dimensional matrices were found  in~\cite{scott}, and some version of this lemma for multidimensional tensors was stated in~\cite{fri}. A removal lemma for $2$-dimensional matrices with respect to a set of submatrices closed under permutations was proved in~\cite{fischer}, and some removal lemmas for row and column ordered appearances of submatrices were recently obtained in~\cite{alonben, alonfish}.  

The main aim of this paper is to give some generalizations of the regularity and counting lemmas for multidimensional matrices.  In Section~\ref{secregmulti}, we prove a variant of the regularity lemma for multidimensional matrices based on a natural translation of $\varepsilon$-regularity from graphs to matrices. In Section~\ref{seccount2}, we state that for $2$-dimensional matrices this notion of regularity is sufficient to obtain a matrix analogue of the counting lemma. Meanwhile, in Section~\ref{seccountmulti}, we provide an example showing that the $\varepsilon$-regularity is too weak for a multidimensional counting lemma. Then we introduce a concept of an $\varepsilon$-regular pattern that will be enough for proving a variant of a counting lemma for multidimensional matrices.

\section{Regularity lemma for multidimensional matrices}\label{secregmulti}

We start the section with definitions required for stating the regularity lemma and with some other basic concepts. 

A \textit{$d$-dimensional matrix $A$ of sizes $n^1 \times \ldots \times n^d$ over an alphabet $\Sigma$} is an array $(a_\alpha)_{\alpha \in I}$, $a_\alpha \in\Sigma$, where the set of indices $I = \left\{ (\alpha_1, \ldots , \alpha_d):\alpha_i \in \left\{1,\ldots,n^i \right\}\right\}$.  We will say that $A$ has \textit{order $n$} if all  sizes $n^i$ equal $n$. If the alphabet $\Sigma$ consists of only two symbols $\{ 0,1\}$ we will say that $A$ is a \textit{binary matrix.}

A \textit{block} $B$ of sizes $m^1 \times \cdots \times m^d$ of a $d$-dimensional matrix $A$ is an arbitrary submatrix of $A$ with such sizes.  

Given symbol $\sigma \in \Sigma$ and a  $d$-dimensional matrix (or a block) $A$ of sizes $n^1 \times \cdots \times n^d$,  the \textit{weight $w_{\sigma}(A)$ of symbol $\sigma$} in $A$ is the number of appearances of $\sigma$ in $A$, the \textit{volume} $|A|$ is $\prod\limits_{i=1}^d n^i$ (number of entries in $A$), and the \textit{density}  $\rho_\sigma(A)$ of the symbol $\sigma$ is
$$\rho_\sigma(A) = \frac{w_\sigma(A)}{|A|}.$$ If $A$ is a binary matrix then we define the \textit{density} $\rho(A)$ to be  the density of symbol  $1$.

A \textit{block partition} of a $d$-dimensional matrix $A$ of order $n$ is a system of  blocks $\mathcal{B} = \{ B_{\beta}\}_{\beta \in J}$, where the index set is $J = \{ (\beta_1, \ldots, \beta_d)| \beta_i \in \{ 1, \ldots, t^i\}  \}$, such that  $B_\beta$ are disjoint blocks  of sizes $m^1_{\beta_1} \times \cdots \times m^d_{\beta_d}$ and the union of all blocks composes the matrix $A$. 
A block partition $\mathcal{C} = \{ C_{\gamma}\}$ of the matrix $A$ is called a \textit{refinement} of a block partition $\mathcal{B} = \{ B_{\beta}\}$ if for each block $B_\beta$ the system of blocks $\mathcal{C}_\beta = \{ C_\gamma | C_\gamma \subset B_\beta \}$ is a block partition of $B_\beta$.

Given $\varepsilon > 0$, a $d$-dimensional  matrix $A$ of sizes $n^1 \times \cdots \times n^d$  over  $\Sigma$   is called \textit{$\varepsilon$-regular} if for each $\sigma \in \Sigma$ and for any block $B$ in $A$ with sizes $m^1 \times \cdots \times m^d$, where $m^i \geq \varepsilon n^i$ for all $i = 1, \ldots, d$, it holds
$$|\rho_\sigma(B) - \rho_{\sigma}(A)| \leq \varepsilon.$$
Otherwise the matrix $A$ is said to be \textit{$\varepsilon$-irregular}.

Before we define an $\varepsilon$-regular block partition, recall that an $\varepsilon$-regular partition of graph vertices consists of a sufficiently small exceptional part and many ``good'' ordinary parts. Acting in a similar way, for a given block partition $\mathcal{B}$  we fix a division of all its blocks into two classes: a set of \textit{ordinary blocks}, whose union composes some submatrix of $A$, and a set of \textit{exceptional blocks}.  Unless otherwise stated, we assume  that a refinement $\mathcal{C}$ of a block partition $\mathcal{B}$ preserves the class of blocks:  if a block $B_\beta$ was exceptional (ordinary) for a block partition $\mathcal{B}$ then all blocks of $\mathcal{C}_\beta$ will be also exceptional (ordinary) in its refinement $\mathcal{C}$.

 Having fixed such a division into ordinary and exceptional blocks, let the \textit{cardinality} $|\mathcal{B}|$  of a block partition $\mathcal{B} = \{ B_{\beta}\}$  be the number of  its ordinary blocks. We will say that the block partition $\mathcal{B}$ is a \textit{balanced} partition of order $t$ if all its ordinary blocks have the same order $m$ and the union of all ordinary blocks is a matrix of order $tm$.  In particular,  the cardinality of every balanced block partition of order $t$ of a $d$-dimensional matrix equals $t^d$.  

We say that a balanced block partition $\mathcal{B} = \{ B_{\beta}\}$ of a $d$-dimensional matrix $A$ is  \textit{$\varepsilon$-regular} if  the sum of volumes of all exceptional blocks is not greater than $\varepsilon |A|$ and all ordinary blocks except, probably, at most $\varepsilon |\mathcal{B}|$ of them are $\varepsilon$-regular. Otherwise, a balanced block partition $\mathcal{B}$ is called \textit{$\varepsilon$-irregular}.

The main result of this section is the following theorem that naturally generalizes the regularity lemma for the case of multidimensional matrices.

\begin{teorema} \label{semmulti}
For every $1/2 > \varepsilon > 0$ there is $N, T \in \mathbb{N}$ such that for every matrix $A$ of order $n \geq N$   over a finite alphabet $\Sigma$ there exists an $\varepsilon$-regular block partition $\mathcal{B} = \{ B_{\beta}\}$  with the cardinality $|\mathcal{B}|$ bounded by $T$.
\end{teorema}

Our proof of the theorem follows the lines of paper~\cite{scott} and is based on a standard (for proofs of regularity lemmas) argument of a substantiation increment of certain function after a suitable refinement of an $\varepsilon$-irregular partition.
For convenience, the proof is divided into a series of lemmas. 

Let us introduce functions $\overline{\varphi}$ and $\tilde{\varphi}$ defined on blocks and block partitions respectively. Given symbol $\sigma \in \Sigma$ and  block $B$ of a block partition $\mathcal{B} = \{ B_{\beta}\}$ of a matrix $A$, we put $\overline{\varphi_\sigma}(B) = \rho^2_\sigma(B) |B| $ if $B$ is an ordinary block and $\overline{\varphi_\sigma}(B) = w_\sigma(B)$ in case of an exceptional block $B$.  Next, we put $\overline{\varphi} (B) = \sum_{\sigma} \overline{\varphi_\sigma} (B)$.  At last, for a block partition $\mathcal{B}$ we introduce function  $\tilde{\varphi} (\mathcal{B}) = \sum_{\beta} \overline{\varphi} (B_\beta)$. Note that for every block $B$  it holds $\rho^2_\sigma(B)|B| \leq w_\sigma(B)$, therefore moving $B$ from  ordinary blocks to the set of exceptional blocks of a block partition $\mathcal{B}$ does not decrease the function $\tilde{\varphi} (\mathcal{B})$. 

We start with a proof a fact that the function $\tilde{\varphi}$ is non-decreasing under refinements.

\begin{lemma}  \label{convmono}
Let $A$ be a multidimensional matrix over an alphabet $\Sigma$, and let a block partition $\mathcal{C} = \{ C_\gamma\}$ of the matrix $A$ be a refinement of a block partition $\mathcal{B} = \{ B_\beta\}$ of $A$. Then $\tilde{\varphi} (\mathcal{C}) \geq \tilde{\varphi} (\mathcal{B})$.
\end{lemma}

\begin{proof}
By the definition of a refinement, each exceptional block $B_\beta$ of $\mathcal{B}$ is divided into a system $\mathcal{C}_\beta$  of exceptional blocks  of $\mathcal{C}$. So, for exceptional blocks $B_\beta$ we have $ \tilde{\varphi} (\mathcal{C}_\beta) = \overline{\varphi}(B_\beta).$

Let us prove that $\tilde{\varphi}(\mathcal{C}_\beta) \geq \overline{\varphi}(B_\beta)$ for each ordinary block $B_\beta$ and its  partition $\mathcal{C}_\beta$ into ordinary blocks. 
By the definition, 
$$\tilde{\varphi}(\mathcal{C}_\beta) = \sum\limits_{C \subset B_\beta} \overline{\varphi} (C) = \sum\limits_{\sigma \in \Sigma} \sum\limits_{C \subset B_\beta} \rho^2_\sigma(C) |C|.$$
Since $f(x) = x^2$ is a convex function and $\sum\limits_{C \subset B_\beta} |C| = |B_\beta|$, the Jensen's inequality gives
\begin{gather*}
\tilde{\varphi}(\mathcal{C}_\beta) \geq \sum\limits_{\sigma} |B_\beta| \left( \sum\limits_{C \subset B_\beta} \frac {\rho_\sigma(C) |C|}{|B_\beta|}  \right)^2 = \sum\limits_{\sigma}  |B_\beta| \left( \sum\limits_{C \subset B_\beta} \frac {w_\sigma(C)}{|B_\beta|}  \right)^2 \\
  = \sum\limits_{\sigma} |B_\beta|  \rho^2_\sigma(B_\beta)  = \sum\limits_{\sigma} \overline{\varphi_\sigma}(B_\beta) = \overline{\varphi} (B_\beta). 
\end{gather*}
\end{proof}

Next, we bound the value of $\tilde{\varphi} (\mathcal{B})$ for any block partition $\mathcal{B}$ of a matrix $A$.

\begin{lemma} \label{boundonphi}
If $\mathcal{B}$ is a block partition of a multidimensional matrix $A$ then $\tilde{\varphi}(\mathcal{B}) \leq |A|$.
\end{lemma}

\begin{proof}
By Lemma~\ref{convmono}, the function $\tilde{\varphi}$ achieves the maximum value on a block partition $\mathcal{S}$  into singleton blocks. It is easy to see that $\tilde{\varphi}(\mathcal{S}) = |A|$.
\end{proof}

The aim of the following lemma is to show that if $B$ is not an $\varepsilon$-regular matrix then there is a block partition $\mathcal{C}$ of $B$ with $\tilde{\varphi}(\mathcal{C})$ substantively greater than $\overline{\varphi}(B)$.

\begin{lemma} \label{increasephi}
Let $ \varepsilon > 0$. Assume that a $d$-dimensional matrix $B$ of order $n$  is $\varepsilon$-irregular. Then there exists a block partition $\mathcal{C}  = \{ C_\gamma\} $  of the matrix $B$ into  $2^d$ ordinary blocks such that
$$\tilde{\varphi}(\mathcal{C}) \geq \overline{\varphi} (B) + \varepsilon^{d+2} |B|.$$
\end{lemma} 

\begin{proof}
Let $\rho_\sigma$ denote the density $\rho_\sigma(B)$ of the matrix $B$.
Since $B$ is $\varepsilon$-irregular, there exists $\tau \in \Sigma$ and a block $C$ of $B$ of sizes $m^1 \times  \cdots \times m^d $ such that $m^i \geq \varepsilon n$ for all $i$ and  $|\rho_\tau(C) - \rho_\tau| \geq \varepsilon$.   

The block $C$ naturally  induces the block partition $\mathcal{C} = \{C_\gamma \}$  of the $d$-dimensional matrix $B$   into $2^d$ ordinary blocks, with the block $C$ being one of blocks of the partition $\mathcal{C}$. For shortness, let $\rho_{\sigma ,\gamma}$ denote the density $\rho_\sigma(C_\gamma)$ of symbol $\sigma$ in the block $C_\gamma$.

For each block $C_\gamma$ we put $\delta_{\sigma, \gamma} = \rho_{\sigma, \gamma} - \rho_\sigma$. 
Let us estimate $\tilde{\varphi}(\mathcal{C})$. By definitions, 
$$\tilde{\varphi}(\mathcal{C}) = \sum\limits_{\sigma} \sum\limits_\gamma |C_\gamma| \rho_{\sigma,\gamma}^2 = \sum\limits_{\sigma} \sum\limits_\gamma |C_\gamma| \left(\rho^2_\sigma + 2 \delta_{\sigma,\gamma} \rho_\sigma + \delta_{\sigma,\gamma}^2 \right).$$

Note that for each $\sigma$ we have $\sum\limits_\gamma |C_\gamma| \rho_\sigma^2 = |B| \rho_{\sigma}^2 = \overline{\varphi_\sigma}(B) $ and  $\sum\limits_{\gamma} \delta_{\sigma,\gamma} |C_\gamma| = 0$ because 
$$\sum\limits_{\gamma} \rho_{\sigma,\gamma} |C_\gamma| =  \sum\limits_{\gamma} w_\sigma(C_\gamma) = w_\sigma(B) = \sum\limits_{\gamma} \rho_\sigma |C_\gamma|. $$

Therefore,
$$\tilde{\varphi}(\mathcal{C}) =  \overline{\varphi}(B) + \sum\limits_{\sigma}\sum\limits_\gamma \delta^2_{\sigma,\gamma} |C_\gamma|  \geq  \overline{\varphi}(B) + \varepsilon^{d+2} |B| , $$
because, by the choice of block $C = C_{\gamma'}$ for some $\gamma'$, we have $|\delta_{\sigma, \gamma'}| \geq \varepsilon$ and $|C_{\gamma'}| = |C| \geq \varepsilon^d |B|$.
\end{proof}

Using this lemma, we prove that if $\mathcal{B}$ is not an $\varepsilon$-regular block partition then there is a refinement $\mathcal{C}$ for which $\tilde{\varphi} (\mathcal{C})$ is  substantively greater than $\tilde{\varphi}(\mathcal{B})$.

\begin{lemma} \label{goodrefine}
Let us fix $\varepsilon > 0$ and let $A$ be a $d$-dimensional matrix of order $n$ over a finite alphabet $\Sigma$. Assume that $\mathcal{B} = \{ B_\beta\} $ is a balanced $\varepsilon$-irregular  block partition of order $t$ of the matrix $A$  with the sum of volumes of all its exceptional blocks  equal to $V \leq \varepsilon |A|$. Then there exists a balanced refinement $\mathcal{C} = \{ C_\gamma\}$ of the block partition $\mathcal{B}$ such that  the cardinality  $|\mathcal{C}| \leq  8^{d^2t}  |\mathcal{B}| $, the sum of volumes all exceptional blocks is not greater than $ V + \frac{d}{2^{dt}} |A|$ and such that 
$$\tilde{\varphi}(\mathcal{C}) \geq \tilde{\varphi}(\mathcal{B}) + \varepsilon^{d+3}(1-\varepsilon)  \cdot |A| .$$
 \end{lemma}

 \begin{proof}
Because the block partition $\mathcal{B}$ is $\varepsilon$-irregular there are at least $\varepsilon |\mathcal{B}|$ ordinary blocks $B_\beta$ of order $m$ that are $\varepsilon$-irregular.   By Lemma~\ref{increasephi}, for each such $\beta$  there exists a refinement $\mathcal{D}_\beta$ of the block  $B_\beta$ into $2^d$ ordinary blocks such that $\tilde{\varphi}(\mathcal{D}_\beta) \geq \overline{\varphi} (B_\beta) + \varepsilon^{d+2} m^d.$
 
Denote by  $\mathcal{D}$  an auxiliary block partition of the matrix $A$ that is a refinement of the block partition $\mathcal{B}$ into a minimal number blocks refining all blocks of partitions $\mathcal{D}_\beta$.  Recall that the class (ordinary or exceptional) of blocks of $\mathcal{D}$ is defined by embracing blocks.
 
Since each ordinary block of $\mathcal{B}$ is divided into at most  $2^{dt}$ ordinary blocks in the refinement $\mathcal{D}$, the cardinality $|\mathcal{D}|$ is not greater than $2^{dt} \cdot |\mathcal{B}|$. Also note that the block partition $\mathcal{D}$ has the same sum of volumes of exceptional blocks as $\mathcal{B}$.
 
 Let us estimate an increment of the function $\tilde{\varphi}(\mathcal{D})$ with respect to $\tilde{\varphi}(\mathcal{B})$. By Lemma~\ref{convmono}, refinements of blocks  do not decrease the function $\overline{\varphi}$. Because there are at least $\varepsilon |\mathcal{B}|$ blocks $B_\beta$ of order $m$ for which $\tilde{\varphi}(\mathcal{D}_\beta) \geq \overline{\varphi} (B_\beta) + \varepsilon^{d+2} m^d$, we have that
 $$\tilde{\varphi}(\mathcal{D}) \geq \tilde{\varphi}(\mathcal{B}) + \varepsilon^{d+3} \cdot |\mathcal{B}| m^d =  \tilde{\varphi}(\mathcal{B}) + \varepsilon^{d+3}  \cdot (tm)^d \geq \tilde{\varphi}(\mathcal{B}) + \varepsilon^{d+3}(1 - \varepsilon) \cdot  |A| .$$ 
 The last inequality holds because $(tm)^d$ is the sum of volumes of all ordinary blocks and the sum of volumes of all exceptional blocks is not greater than $ \varepsilon |A|$.
 
 Let $\mathcal{C} = \{ C_\gamma\}$ be a balanced block partition of $A$ with ordinary blocks of order $l = \frac{n}{t 4^{dt}}$  such that $\mathcal{C}$ is  the minimal (on a number of blocks) refinement of the block partition $\mathcal{D}$, with each ordinary block of $\mathcal{D}$ containing as much as possible ordinary $l$-ordered blocks of the partition $\mathcal{C}$.   All other blocks of the partition $\mathcal{C}$ are set to be exceptional.
 
 Conditions on the block partition $\mathcal{C}$ imply that  the sum of volumes of all exceptional blocks $C_\gamma$ within each ordinary block of $\mathcal{D}$ is not greater than $d l m^{d-1}$, so the sum of volumes of all exceptional blocks of the partition $\mathcal{C}$ does not exceed
 $$V + d l m^{d-1}\cdot  |\mathcal{D}| \leq   V +  d \cdot   \frac{n}{t 4^{dt}}  \cdot  m^{d-1} \cdot  2^{dt}   t^d  \leq V +  \frac{d}{2^{dt}} \cdot n (tm)^{d-1}  \leq V + \frac{d}{2^{dt}} \cdot  |A|. $$
 
 On the other hand, the cardinality of the block partition $\mathcal{C}$ is 
 $$|\mathcal{C}| \leq  \frac{m^d}{l^d}  |\mathcal{D}|  = \frac{(tm)^d}{n^d} 4^{d^2t} \cdot 2^{dt} |\mathcal{B}|   \leq 8^{d^2t}  |\mathcal{B}|. $$ 
 
It only remains to estimate the function $\tilde{\varphi}(\mathcal{C})$. Using Lemma~\ref{convmono}, we obtain
$$\tilde{\varphi}(\mathcal{C}) \geq \tilde{\varphi} (\mathcal{D}) \geq\tilde{\varphi}(\mathcal{B}) + \varepsilon^{d+3}(1-\varepsilon) \cdot |A|. $$
 \end{proof}

Now we are ready to prove the main result of this section.

\begin{proof}[Proof of Theorem~$\ref{semmulti}$]
Let $\mathcal{D}$ be an arbitrary balanced block partition of order $t$ of the matrix $A$ with the sum of volumes of all exceptional blocks equal to $V$, where  $0 < V < ( \varepsilon - 2 \cdot \frac{d}{2^{dt}} )|A| $. Let us iteratively apply Lemma~\ref{goodrefine} to the partition $\mathcal{D}$ until we get an $\varepsilon$-regular block partition. The inequality on $V$ ensures that the applications of Lemma~\ref{goodrefine}  do not make   the sum of volumes of exceptional blocks greater than $\varepsilon |A|$. 

After each iteration the function $\tilde{\varphi} $  increases by at least $\varepsilon^{d+3}(1-\varepsilon) \cdot |A|$. By Lemma~\ref{boundonphi}, for any block partition  the function $\tilde{\varphi}$  does not exceed $|A|$. Thus,  after at most $\varepsilon^{-d-3}(1 - \varepsilon)^{-1}$ applications of Lemma~\ref{goodrefine} we  get an $\varepsilon$-regular block partition $\mathcal{B}$. Lemma~\ref{goodrefine} also implies that the cardinality of the obtained block partition $\mathcal{B}$ can be bounded by some constant depending only on $\varepsilon$.

Therefore, if the matrix $A$ has a large enough order then we can always find its $\varepsilon$-regular block partition whose cardinality bounded by $\varepsilon$.

\end{proof}

\section{Counting lemma for $2$-dimensional matrices} \label{seccount2}

In this section we prove a generalization of the counting lemma (Lemma~\ref{countgr}) for $2$-dimensional matrices over a finite alphabet.  

Given $2$-dimensional matrices $A$ and $C$ over $\Sigma$, we will say that $A$ contains the matrix $C$ if after appropriate permutations of rows and columns of $A$ one can find a copy of the matrix $C$ in the resulting matrix. In other words, we will say that $A$ contains the matrix $C$ if there is at least one submatrix of $A$ obtained from $C$ by row and column permutations. 
In what follows,  the number of matrices $C$ in the matrix $A$ means the number of all distinct appearances of row and column permutations of $C$ in the matrix $A$.

The following theorem states that having a partition of a $2$-dimensional matrix $A$ into $\varepsilon$-regular blocks we can find many given submatrices $C$ in the matrix $A$.

\begin{teorema} \label{count2dim}
For any $\delta > 0$, $t,s \in \mathbb{N}$ and $0 < \rho_{i,j} (\sigma) \leq 1$  there is $\varepsilon > 0$ such that the following holds. If  $\mathcal{B} = \{ B_{i,j} \}$ is a  block partition  of a $2$-dimensional matrix $A$ over $\Sigma$ into blocks of sizes $m_{i} \times n_j$, $i = 1, \ldots, t;$ $j = 1, \ldots, s$ such that all  $B_{i,j}$  are $\varepsilon$-regular blocks having density of a symbol $\sigma$ equal to $\rho_{i,j}(\sigma)$ then the number of submatrices $C$ of sizes $t \times s$ in the matrix $A$ is at least 
$$ \left( 1 - \delta \right) \prod\limits_{i=1}^t m_i \prod\limits_{j=1}^{s} n_j  \rho_{i,j}(c_{i,j}).$$
\end{teorema}

For the sake of convenience, let us define two special types of blocks. 
Given a $2$-dimensional matrix $A$, let a \textit{v-line of length $l$} (vertical line)  to be a block of sizes $l \times 1$, and an \textit{h-line of length $l$} (horizontal line) to be a block of sizes $1 \times l$. V-lines of the maximal length are exactly \textit{columns} of the matrix $A$ and h-lines of the maximal length are  \textit{rows} of $A$.
 
 We start the proof of Theorem~\ref{count2dim} with some auxiliary lemma on properties of $\varepsilon$-regular matrices.

\begin{lemma} \label{regularline}
Suppose that $B$ is an $\varepsilon$-regular matrix of sizes $m \times n$ over $\Sigma$ and of density $\rho_\sigma$ of symbols $\sigma$. Let $U$ be a set of  v-lines  (or h-lines) of $B$ of length $l \geq \varepsilon m$ (or of length $l \geq \varepsilon n$) composing a block of $B$.  Then the number  of v-lines (h-lines) of $U$ with the density of a symbol $\sigma$  strictly less than $\rho_\sigma - \varepsilon$ is less than $\varepsilon n$ (less than $\varepsilon m$).
\end{lemma}

\begin{proof}
Assume that there are at least $\varepsilon n$ v-lines of $U$ with the density of $\sigma$ less $\rho_\sigma - \varepsilon$ and let $C$ be a block formed by all such v-lines. Note that  the density $\rho_\sigma(C)$ of symbol $\sigma$ in $C$ is also less than $\rho_\sigma - \varepsilon$. On the other hand, all sizes of the block $C$ are at least $\varepsilon$-fractions of sizes of the matrix $B$. The definition of $\varepsilon$-regularity implies that $|\rho_\sigma(C) - \rho_\sigma| \leq \varepsilon$, so $\rho_\sigma(C) \geq \rho_\sigma - \varepsilon$: a contradiction.

The proof for the case of h-lines is similar. 
\end{proof}

Using this lemma, we prove Theorem~\ref{count2dim}.

\begin{proof}[Proof of Theorem~$\ref{count2dim}$]
The main idea of the proof is to construct blocks $C$ in the matrix $A$ by taking entries $c_{i,j}$ from  blocks $B_{i,j}$.  For this purpose, we use induction on the number $s$ of columns in a block partition of $A$.

By Lemma~\ref{regularline}, for each block $B_{i,1}$ the number of columns with density of symbol $c_{i,1}$ at least $\rho_{i,1}(c_{i,1}) - \varepsilon$ is not less than $(1 - \varepsilon)n_1$. So we can choose a block $\mathcal{V}_1$ composed by  $V_{1,1}, \ldots, V_{t,1}$, where  each $V_{i,1}$ is a v-line of $B_{i,1}$ filled by  symbols $c_{i,1}$ and the length of $V_{i,1}$ is not less than $(\rho_{i,1}(c_{i,1}) - \varepsilon) m_i$. Note that the number of columns of the matrix $A$ in which we can find such a block $\mathcal{V}_1$ is at least $(1 - t\varepsilon ) n_1.$  Therefore, the number of appearances of the first column of $C$ within the first column of the partition $\mathcal{B}$ is at least  $(1 -t \varepsilon) n_1  \prod\limits_{i=1}^t (\rho_{i,1}(c_{i,1}) - \varepsilon) m_i. $

Suppose that we have already constructed at least  $(1 - t \varepsilon)^{k-1}  n_1  \cdots n_{k-1}$ different blocks $\mathcal{V}_{k-1}$ formed by v-lines  $V_{i,j}$ of blocks $B_{i,j}$ filled by symbols $c_{i,j}$. We also assume that the length of each v-line $V_{i,j}$ is at least $ m_i \prod\limits_{j=1}^{k-1}\left( \rho_{i,j}(c_{i,j}) - \varepsilon\right)$. Consider v-lines in blocks $B_{i,k}$ that share exactly the same rows in the matrix $A$ as columns of the block $\mathcal{V}_{k-1}$. By Lemma~\ref{regularline}, for each block $B_{i,k}$ the number of such v-lines with density of symbol $c_{i,k}$ at least $\rho_{i,k}(c_{i,k}) - \varepsilon$ is not less than $(1 - \varepsilon)n_{k}$.   So with the help of these lines we can extend a block $\mathcal{V}_{k-1}$ to the block $\mathcal{V}_{k}$ satisfying the similar conditions. By conditions, the length of columns in the block $\mathcal{V}_k$  is not less than $ m_i \prod\limits_{j=1}^k (\rho_{i,j}(c_{i,j}) - \varepsilon)$.

Note that a given block $\mathcal{V}_{k-1}$ has  at least $(1 - t\varepsilon ) n_{k}$ expansions to the block $\mathcal{V}_k$.   Consequently, we construct at least $(1 - t \varepsilon)^{k} n_1  \cdots n_{k}$ blocks $\mathcal{V}_{k}$.  Moreover, the number of  appearances of the first $k$ columns of $C$ in the block $\mathcal{V}_k$ containing  is at least $(1 -t \varepsilon)^{k} \prod\limits_{i=1}^t m_i \prod\limits_{j=1}^{k} n_j (\rho_{i,j}(c_{i,j}) - \varepsilon). $
 
 Iterating this process until $k = s$, we obtain that the number of blocks $C$ in the matrix $A$ is at least 
 $$ (1 - t \varepsilon )^s  \prod\limits_{i=1}^t m_i \prod\limits_{j=1}^{s} n_j  (\rho_{i,j}(c_{i,j}) - \varepsilon).$$
It remains to note that if $\varepsilon = \varepsilon(\delta, t, s, \rho_{i,j}(c_{i,j}))$ is small enough then the number of blocks $C$ in  the matrix $A$ is greater than
 $$\left( 1 - \delta \right)  \prod\limits_{i=1}^t m_i \prod\limits_{j=1}^{s} n_j \rho_{i,j}(c_{i,j}).$$
 \end{proof}

In the statement of Theorem~\ref{count2dim} we can permute rows and columns of the block partition $\mathcal{B}$ (or rows and columns of the block $C$) and get many lower bounds on the number of appearances of $C$. Moreover, to obtain a better bound we can sum these bounds over all different placements of the matrix $C$.

\section{Counting lemma for multidimensional matrices} \label{seccountmulti}

We start this section with an example showing that the statement of the $2$-dimensional counting lemma fails in a multidimensional case: for every $\varepsilon > 0$ there exist multidimensional binary matrices of arbitrary large order formed by $\varepsilon$-regular blocks of nonzero density such that they do not contain some given submatrix.

\begin{utv} \label{contrex}
For each $d \geq 3$ and every $0 < \varepsilon <1$ there is $N = N(\varepsilon)$ such that for all $n \geq N$ there exist a $d$-dimensional binary matrix $A$ of order $n$, a balanced block partition $\mathcal{B} = \{ B_\beta\}_{\beta \in J}$ and a binary matrix $U = \{u_\beta\}$ satisfying the following property. Each block $B_\beta$ is $\varepsilon$-regular and has density $\rho = 1/2$ but there are no submatrices $U$ in the matrix $A$.
\end{utv}

\begin{proof}
We prove the proposition for $3$-dimensional matrices $A$ and a block partition $\mathcal{B} = \{B_\beta \} $ containing $8$ blocks $B_\beta$, where index $\beta = (\beta_1, \beta_2, \beta_3)$, $\beta_i \in \{ 0,1\}$. Also we use some binary matrix $U$ of order $2$ having the following entries:
$$u_{0,0,0} = u_{0,0,1}= u_{1,1,0}=1; ~~~~ u_{1,1,1} = 0.$$
 One can easily extend this construction for greater dimensions and other block partitions.

Let $H = \{ h_{i,j}\}$ be an arbitrary $2$-dimensional $\varepsilon$-regular binary matrix of order $n$ and of density $\rho = 1/2$. Such a  matrix $H$ can be obtained, for example, via a little modifications of the Hadamard matrices and replacements of $-1$s by $0$s. The fact that the density of $1$s (or $-1$s) in each large enough block of a Hadamard matrix is close to $1/2$ was firstly proved in~\cite{alon}.

Let us define entries of blocks $B_\beta$ by the following way: if the sum $\beta_1 + \beta_2 + \beta_3$  is even then we put $b_{i,j,k} = h_{i,j}$, for all other blocks $B_\beta$ let  $b_{i,j,k} = 1- h_{i,j}$.

By the construction of blocks $B_\beta$ and by properties of the matrix $H$, for every block  $C$ in  $B_\beta$ with sizes $m_1 \times m_2 \times m_3$, $m_i \geq \varepsilon n$ it holds
$$|\rho (B_\beta) - \rho(C)| = |1 / 2 - \rho(C)| \leq  \varepsilon .$$
Therefore, all blocks $B_\beta$ are $\varepsilon$-regular.

It remains to prove that there are no blocks $U$   in the matrix $A$. The union of all blocks $B_{i,j,0}$ (or, respectively, the union of $B_{i,j,1}$) does not contain the matrix $U$ because   for all entries of such $B_\beta$ we have $b_{i,j,k} = b_{i,j,l}$, while $u_{1,1,0} \neq u_{1,1,1}$.   Without loss of generality, assume now that $U$ uses entries from blocks $B_{0,0,0}$ and $B_{0,0,1}$ and  let $b_{0,0,0} =h_{0,0} = 1$ be an entry from the block $B_{0,0,0}$ belonging to the submatrix $U$. Since $U$ should also contain an unity entry from the block $B_{0,0,1}$, there is some $k$ such that the entry $b'_{0,0,k} = 1 - h_{0,0}$ from the block $B_{0,0,1}$ is equal to 1.  So we have that $h_{0,0}$ is simultaneously equal to $0$ and $1$: a contradiction. 

At last, permutations of hyperplanes of the matrix $U$ also produce matrices that cannot be found within the matrix $A$.
\end{proof}

Proposition~\ref{contrex} implies that for obtaining a multidimensional analogue of the counting lemma we need a stronger regularity of block partitions. So we introduce the concept of $d$-dimensional $\varepsilon$-regular patterns. In some sense, they are close to $\langle \delta \rangle $-regular hypergraphs that were proposed recently to prove the tightness of hypergraph regularity lemmas in~\cite{tight}.
 
Let $0 < \varepsilon, \rho_\sigma <1$ and $n, d \in \mathbb{N}$. Our definition of $\varepsilon$-regular patterns is inductive on $d$.    A \textit{$1$-dimensional $\varepsilon$-regular pattern} of order $n$ is an arbitrary $1$-dimensional matrix of order $n$ over the alphabet $\Sigma$  and with densities of symbols $\sigma$ equal $ \rho_{\sigma}$.  Recall that matrices (and patterns) over the alphabet $\{ 0,1\}$ are called binary.
  
Define a \textit{$d$-dimensional $\varepsilon$-regular pattern} $A$ of order $n$ and densities $\rho_\sigma$ to be a $d$-dimensional matrix of order $n$ such that for each $(d-1)$-dimensional binary $\varepsilon$-regular  pattern of order $n$  and density $\rho' \geq \varepsilon$ its entrywise multiplication with at least $(1 - \varepsilon)n$ hyperplanes of $A$ of each direction is a $(d-1)$-dimensional $\varepsilon$-regular pattern  with densities of $\sigma$ at least $\rho_\sigma \rho'(1 - \varepsilon)$. Here a \textit{hyperplane} of a $d$-dimensional matrix is a maximal $(d-1)$-dimensional block in this matrix, and under a \textit{direction} of a hyperplane we mean an index position fixed in all its entries. 
  
To show that the set of multidimensional $\varepsilon$-regular patterns is nonempty, let us prove that a random matrix is likely to be an $\varepsilon$-regular pattern.

 \begin{utv}
 Let $A$ be a random $d$-dimensional matrix of order $n$ over an alphabet $\Sigma$, where entries  $a_\alpha$ are random variables taking value $\sigma$ with probability $\rho_\sigma$. Then for every given $\varepsilon >0$ the matrix $A$ is a $d$-dimensional $\varepsilon$-regular pattern with high probability as $n \rightarrow \infty$.
 \end{utv}

 \begin{proof}
 
 Firstly we note that each hyperplane of a such $d$-dimensional matrix is a $(d-1)$-dimensional random matrix.
 
 Let us prove that with probability tending to $1$ the entrywise product $B$ of some $\varepsilon$-regular $d$-dimensional binary pattern $P$ and a random matrix over $\Sigma$  is an $\varepsilon$-regular $d$-dimensional pattern over $\Sigma$. Then the proposition follows from the fact that every random matrix is the entrywise product of the  $\varepsilon$-regular pattern whose entries are all $1$ and the matrix itself.
 
 We prove the statement by induction on $d$. For the $1$-dimensional case it is a direct consequence of the definition of $\varepsilon$-regular patterns and the law of large numbers.

 Let $B$ be $d$-dimensional matrix of order $n$ over $\Sigma$ obtained by the entrywise product of a random  matrix and a $d$-dimensional binary  $\varepsilon$-regular pattern $P$. Consider an arbitrary  $(d-1)$-dimensional binary $\varepsilon$-regular pattern $L$. Since $P$ is an $\varepsilon$-regular pattern, the entrywise product of at least $(1 - \varepsilon)n$ hyperplanes $\Gamma$ of $P$ of some direction and the pattern $L$ is a $(d-1)$-dimensional  binary $\varepsilon$-regular  pattern $L_\Gamma$. By the inductive assumption, the probability that the  entrywise product of the binary pattern $L_\Gamma$ and a random matrix  is a $(d-1)$-dimensional $\varepsilon$-regular pattern approaches to $1$ for large $n$.   Therefore, we have  that the definition of $d$-dimensional $\varepsilon$-regular  patterns is satisfied for the matrix $B$ with probability tending to $1$ as $n \rightarrow \infty$. 
 \end{proof}

As the last result of this paper, we state the multidimensional counting lemma. As in Theorem~\ref{count2dim}, when we count the number of appearances of $C$ in a matrix $A$, we look at the number of all submatrices of $A$ that can be obtained by from $C$ by permutations of hyperplanes.  The proof of the present theorem also follows ideas of Theorem~\ref{count2dim}.

\begin{teorema} \label{countmultidim}
For any $\delta > 0$, $t,s \in \mathbb{N}$ and $0 < \rho_{\beta} (\sigma) \leq 1$  there is $\varepsilon > 0$ such that the following holds. 
Let $\mathcal{B} = \{ B_{\beta} \}$ be a  block partition  of a $d$-dimensional matrix $A$ over $\Sigma$ into blocks of sizes $m^1_{\beta_1} \times  \cdots \times m^d_{\beta_d}$, $\beta_j= 1, \ldots, t_j$. If all  $B_{\beta}$  are $\varepsilon$-regular patterns of densities $\rho_{\beta}(\sigma)$ of symbol $\sigma \in \Sigma$, then the number of appearances of a submatrix $C$ of sizes $t_1 \times \cdots \times t_d$ in the matrix $A$ is at least 
$$ \left( 1 - \delta \right) \cdot \prod\limits_{\beta_1=1}^{t_1} m^1_{\beta_1} \cdots  \prod\limits_{\beta_d=1}^{t_d} m^d_{\beta_d} \cdot \rho_{\beta} (c_{\beta}).$$
\end{teorema}

\begin{proof}
In what follows, $\overline{\beta} = (\beta_1 , \ldots, \beta_{d-1})$ denotes the truncation of index $\beta = (\beta_1, \ldots, \beta_d)$ by the $d$-th direction.

The proof is by induction on $d$. The base case ($d= 2$) is exactly Theorem~\ref{count2dim}. 
For greater dimensions,  the main idea of the proof is to count the number of submatrices $C$ in $A$ such that each block $B_\beta$ contains exactly one entry $c_\beta$ of $C$. So we may assume that $B_\beta$ are binary blocks of densities $\rho_\beta = \rho_\beta(c_\beta)$ and we may look for many unity submatrices (having only 1-entries) instead of counting a number of a given submatrix $C$ over $\Sigma$ in $A$. Note that new blocks $B_\beta$ are still $\varepsilon$-regular patterns, possibly, with slightly different $\varepsilon$.

Let us construct many block systems $\mathcal{V} = \{ V_\beta\}$ such that each $V_\beta \subset B_\beta$ is a  $(d-1)$-dimensional binary $\varepsilon$-regular  pattern of density $(1 - \varepsilon)^{t_d} \prod\limits_{i = j}^{t_d} \rho_{\overline{\beta}, j}$ and  $V_\beta = V_{\beta'}$ for all indices $\beta$ and $\beta'$ with $\overline{\beta} = \overline{\beta'}$.

We accomplish this purpose inductively on $\beta_d$. Since all $B_\beta$ are binary   $\varepsilon$-regular patterns,  in each block $B_{\overline{\beta}, 1}$ there exist at least $(1- \varepsilon) m^d_1$ hyperplanes of the $d$-th direction that are binary $\varepsilon$-regular  patterns with density $\rho_{\overline{\beta}, 1}$.  Let a block system $\mathcal{V}_1$ be composed of all such hyperplanes. The number of ways to construct  the block system $\mathcal{V}_1$ is at least $(1 - T \varepsilon)m_1^d$, where $T =t_1 \cdots t_{d-1}.$ 

Assume that we have already constructed the block system $\mathcal{V}_{k-1}$ consisting of $(d-1)$-dimensional binary $\varepsilon$-regular  patterns $V_{\beta} $ in blocks $ B_\beta$ for indices $\beta$ with  $\beta_d = 1, \ldots, k-1$ and satisfying the following properties: density of each pattern $V_\beta$ is at least  $ (1 - \varepsilon)^{k-2} \prod\limits_{j=1}^{k-1} \rho_{\overline{\beta}, j}$ and  $V_\beta = V_{\beta'}$ if $\overline{\beta} = \overline{\beta'}$.   Moreover,  we suppose that there are at least  $(1 - T \varepsilon)^{k-1} m_1^d \cdots m^d_{k-1}$  different block systems $\mathcal{V}_{k-1}$ having these properties.  By the definition of binary $\varepsilon$-regular  patterns, for each block $B_{\overline{\beta}, k}$ there are at least $(1- \varepsilon) m^d_k$  hyperplanes of the $d$-th direction whose entrywise product with the block $V_{\overline{\beta}, j}$ is a $(d-1)$-dimensional $\varepsilon$-regular  pattern with density at least  $ (1 - \varepsilon)^{k-1} \prod\limits_{j=1}^{k} \rho_{\overline{\beta}, j}$. So we can continue the system of blocks $\mathcal{V}_{k-1}$ to a system of blocks $\mathcal{V}_{k}$ satisfying the similar conditions and such that   the density of all new blocks $V_{\beta}$  is not less than $ (1 - \varepsilon)^{k-1} \prod\limits_{j=1}^{k} \rho_{\overline{\beta}, j}$.
Note that for a given system of blocks $\mathcal{V}_{k-1}$ the number of its possible continuations is at least $(1 - T\varepsilon) m^d_{k}.$  Consequently, we have constructed at least $(1 - T \varepsilon)^{k} m^d_1  \cdots m^d_{k}$ unity blocks $\mathcal{V}_{k}$.  

Iterating the above process until $k = t_d$ gives us the required system of blocks $\mathcal{V} = \mathcal{V}_{t_d}$.
 Let $\overline{\mathcal{V}} = \{ V_{\overline{\beta}} \} $ be a system of $(d-1)$-dimensional blocks obtained as a truncation of the $\mathcal{V}$ by the $d$-dimension.  By the construction and by the inductive assumption,  the number of $(d-1)$-dimensional unity blocks of sizes $t_1 \times \cdots \times t_{d-1}$ in the multidimensional matrix formed by blocks of the system $\overline{\mathcal{V}}$ is not less than 
 $$\left( 1 - \delta' \right) \prod\limits_{\beta_1=1}^{t_1} m^1_{\beta_1} \cdots  \prod\limits_{\beta_{d-1}=1}^{t_{d-1}} m^{d-1}_{\beta_{d-1}} (1 - \varepsilon)^{t_d -1}   \prod\limits_{\beta_d=1}^{t_d} \rho_{\beta}$$
 for some $\delta' = \delta'(\varepsilon)$.

Each $(d-1)$-dimensional unity block of sizes  $t_1 \times \cdots \times t_{d-1}$ in the system $\overline{\mathcal{V}}$ easily expands to a $d$-dimensional unity block in the system $\mathcal{V}$ of sizes $t_1 \times \cdots \times t_{d}$. Since there are at least  $(1 - T \varepsilon)^{t_d} m^d_1  \cdots m^d_{t_d}$  different block systems $\overline{\mathcal{V}}$, the total number of unity blocks  with sizes  $t_1 \times \cdots \times t_{d}$ over the block partition  $\mathcal{B}$  is at least
$$\left( 1 - \delta' \right)(1 - \varepsilon)^{t_d -1}  (1 - T \varepsilon)^{t_d} \prod\limits_{\beta_1=1}^{t_1} m^1_{\beta_1} \cdots    \prod\limits_{\beta_d=1}^{t_d} m^d_{\beta_d} \rho_{\beta} \geq \left( 1 - \delta \right)  \prod\limits_{\beta_1=1}^{t_1} m^1_{\beta_1} \cdots    \prod\limits_{\beta_d=1}^{t_d} m^d_{\beta_d} \rho_{\beta}$$
for small enough $\varepsilon$.
 \end{proof}

 As in Theorem~\ref{count2dim},  to maximize the lower bound we can consider hyperplane permutations of the matrix $C$   and take the sum of bounds over all different placements of $C$ in the block partition $\mathcal{B}$. 
 
At last, we note that in the proof Theorem~\ref{countmultidim} we use regularity of blocks only with respect to patterns contained in other blocks of the matrix $A$. So, to get Theorem~\ref{countmultidim}  for a given matrix $A$ and its block partition $\mathcal{B}$ we can weaken the definition of $\varepsilon$-regular patterns by demanding regularity only with respect to patterns that can be found in the blocks of $\mathcal{B}$.

\end{document}